\documentclass[letterpaper, 10 pt, conference]{ieeeconf}  
\usepackage[T1]{fontenc}
\usepackage{amsmath,amsfonts}
\usepackage{color}
\usepackage{graphicx}
\usepackage{graphicx} 
\usepackage{subfigure}
\usepackage{multirow}
\usepackage{makecell}
\usepackage{array}
\usepackage{booktabs}
\newtheorem{theorem}{Theorem}

\IEEEoverridecommandlockouts                              

\title{\LARGE \bf
Sequential Charging Station Location Optimization under Uncertain Charging Behavior and User Growth
}

\author{Wenjia Shen$^{1}$, Bo Zhou$^{2}$, Ruiwei Jiang$^{2}$, and Siqian Shen$^{2}$
\thanks{$^{1}$Wenjia Shen is with the School of Management \& Engineering, Nanjing University, Nanjing, China.
        {\tt\small wenjia150033@gmail.com}}%
\thanks{$^{2}$Bo Zhou, Ruiwei Jiang, and Siqian Shen are with the Department of Industrial \& Operations Engineering, University of Michigan, Ann Arbor, MI U.S.A.
        {\tt\small \{bozum,ruiwei,siqian\}@umich.edu}}%
}

\begin{document}

\maketitle
\thispagestyle{empty}
\pagestyle{empty}

\begin{abstract}
Charging station availability is crucial for a thriving electric vehicle market. Due to budget constraints, locating these stations usually proceeds in phases, which calls for careful consideration of the (random) charging demand growth throughout the planning horizon. This paper integrates user choice behavior into two-stage and multi-stage stochastic programming models for intracity charging station planning under demand uncertainty. We derive a second-order conic representation for the nonlinear, nonconvex formulation by taking advantage of the binary nature of location variables and propose subgradient inequalities to accelerate computation. Numerical results demonstrate the value of employing multi-stage models, particularly in scenarios of high demand fluctuations, increased demand dispersion, and high user sensitivity to the distance-to-recharge.
\end{abstract}
\section{INTRODUCTION}
A transition from conventional fuel vehicles to electric vehicles (EVs) is underway to meet carbon emission targets \cite{Canada}. 
Although policy efforts and market trends are supporting a promising outlook, wider EV adoption is still hindered by insufficient availability of charging facilities \cite{kchaou2021charging}. 
Thus, strategic location planning is essential for supporting an expanding EV market. Given the high cost of charger procurement and power grid expansion \cite{kabli2020stochastic}, the planning is typically long-term and sequential. In each phase of time, only a small batch of stations is constructed to achieve a balance between market coverage and budget limitations of the investors.

In sequential charging station location, it is imperative to not only consider the current demand but also prepare for future charging needs. 
In a dynamic and evolving market, future EV charging demand grows in a stochastic manner \cite{Canada}. 
Although there have been extensive analyses on multi-year planning of intracity stations, research incorporating stochastic demand growth remains limited. Most studies are built on deterministic scenarios where future demand is known \cite{lin2019multistage,lamontagne2023optimising} or calculated using external data such as surveys \cite{vashisth2022multi} and historical charging records \cite{zhang2023multi}. Existing stochastic extensions either consider daily power load fluctuations rather than long-term user growth \cite{wang2017stochastic}, or only employ two-stage models without extending to a multi-stage decision framework \cite{kabli2020stochastic}. Furthermore, the aforementioned studies rarely address the user choice among different stations or the competition between home charging and public charging options.

To fill in the aforementioned gaps, we depict the choice of EV users patronizing charging stations using a multinomial logit (MNL) model. Then, we integrate MNL into a multi-stage stochastic program considering random demand growth and sequential location of charging stations. This yields a fractional integer programming formulation, which is nonlinear and nonconvex. Nevertheless, we derive a second-order conic representation by taking advantage of the binary nature of the location variables. As compared to the state-of-the-art, this drastically reduces the formulation size and accelerates its computation. Furthermore, we conduct numerical experiments based on real EV demand data to demonstrate the value of our model and the effectiveness of our solution approach. We summarize the main contributions as follows.
\begin{enumerate}
    \item We propose the first multi-stage stochastic programming models for sequentially locating charging stations, incorporating random demand growth and user choice behavior.
    \item We derive a second-order conic representation, accompanied by subgradient inequalities, to tackle the ensuing fractional integer programming formulation.
    \item We conduct extensive numerical experiments to demonstrate the value of considering a multi-stage model as opposed to a two-stage one, as well as the effectiveness of the proposed solution approach.
\end{enumerate}

The remainder of this paper is organized as follows. Section \ref{sec:review} reviews the relevant literature, Section \ref{sec:problem} describes our models, and Section \ref{sec:solve} presents the solution approach. Experiment results are reported in Section \ref{sec:numerical} and conclusions are drawn in Section \ref{sec:conclusion}.

\section{LITERATURE REVIEW}\label{sec:review}
An extensive literature has emerged regarding charging station location. We focus on long-term models in Section \ref{subsec:review1} and EV user choice behaviors in Section \ref{subsec:review2}.

\subsection{Long-term Charging Station Location}\label{subsec:review1}
Previous research on long-term planning for charging stations can be broadly classified into two streams: intercity and intracity models. Research in the first stream focuses on the demands arising from the flow of origin-destination city pairs, predominantly establishing charging stations along highways to alleviate ``range anxiety'' \cite{kchaou2021charging,song2023learning}. The multi-period charging station location model was proposed in \cite{chung2015multi}, aiming at maximizing the total covered flow of the Korean expressway network throughout the planning horizon. Subsequent studies have expanded to include annually more cities joining the EV adoption network \cite{li2016multi} and the placement of fast-charging stations \cite{zhang2017incorporating}. More recent research has incorporated parameter uncertainties in long-term planning, such as uncertain driving range, driver behavior, and operational costs \cite{davidov2017planning}, as well as the yearly fluctuations in intercity traffic flow \cite{xie2018long,kadri2020multi}.

Research in the second stream adopts facility location models, where the demand arises from daily intracity activities and is associated with specific nodes of a city network \cite{kchaou2021charging, zhang2021values}. These node-based models are more appropriate when deployment decisions are constrained to an urban setting, where driver movement is typically limited around their assigned locations \cite{anjos2020increasing}. For example, \cite{lin2019multistage} formulated a mixed-integer second-order conic program for locating e-bus charging stations, addressing bus charging needs at depots. \cite{zhang2023multi} integrated queuing models and diverse market scenarios to enhance the flexibility and applicability of charging stations. Subsequently, the model was extended by incorporating more objectives, including proposing multiple solutions for different investor concerns, reducing initial investment, or maximization of EV adoption among residents \cite{lamontagne2023optimising,vashisth2022multi}. 
These studies, however, often depicted a deterministic context where user demand is either known or calculated using external data. Few relevant studies, e.g., \cite{wang2017stochastic}, captured the stochastic daily charging loads and suggested an integrated model for charging stations and power distribution, yet it did not address the long-term (e.g., annual) stochastic demand growth. Ref.~\cite{kabli2020stochastic} employed Monte Carlo simulation for demand scenario generation and developed a two-stage stochastic model. In contrast to these studies and to make applicable for the evolving EV market, this paper focuses on the long-term growth of stochastic charging demands.

\subsection{EV User Choice Behavior}\label{subsec:review2} 
User preferences affect the charging demand of each station, but this has not been fully addressed in the current literature.
In the intercity studies (e.g., \cite{chung2015multi,li2016multi}), users only choose the stations along the travel paths. In most intracity studies, the decision-making process is often 
omitted and the demand at each station is directly specified (see, e.g., \cite{vashisth2022multi,wang2017stochastic}). Alternatively, it is determined based on simplified criteria, such as selecting the nearest station \cite{zhang2023multi} or being restricted to stations within the same geographical grid \cite{kabli2020stochastic}. User behavior depiction was only considered in \cite{lamontagne2023optimising}, which applied a discrete choice model in patronizing charging stations. Unfortunately, their approach did not scale well to large-scale instances, for which they applied heuristics.

Facility location models typically characterize user behaviors by the utility of facilities in a discrete choice model \cite{qi2022sequential}. Existing research primarily divides into deterministic and probabilistic choice models. The former assumes that users always go for the facility with the highest utility (e.g., \cite{lamontagne2023optimising,gentile2018integer}), while the latter splits demand across multiple facilities with a certain probability (e.g., \cite{sun2016fast}), including the MNL model. For EV charging, MNL has been empirically validated by the data from California~\cite{lee2020exploring} and German EV users~\cite{anderson2023real}, 
which explored the impact of, e.g., user characteristics, facility features, and facility availability. In this paper, we adopt the MNL model to describe the charging behavior of EV users probabilistically.

\section{PROBLEM FORMULATIONS}\label{sec:problem}

\subsection{Deterministic Multi-period Model}\label{model:deter}
Over a planning horizon (a set \(T\) of years), an investor seeks to deploy charging stations in a city to serve the EV users located in a set $I$ of nodes and maximize the total revenue. In each year $t \in T$, the charging demand is $d_i^t$ at each node $i \in I$. To construct a new station, the investor selects a site $j$ from a set \(J\) of candidate locations and chooses a type $k \in K$. The set $K$ includes various combinations of station scale levels and characteristics (e.g. fast versus slow charging options). The construction cost for a station is $c_{jk}$ and the unit revenue of a type $k$ station serving a unit of demand is $r_k$. The investor operates within an annual budget of $B^t$, and the location decision is captured by a binary variable $x_{jk}^t$, such that $x_{jk}^t = 1$ if and only if we locate a charging station of type $k$ at location $j$ in year $t$. Finally, any existing facilities at time \(0\) is indicated by $x_{jk}^0$.

To depict the EV user's charging choice, we follow MNL~\cite{mcfadden1973conditional} and define the utility of a user at node \(i\) to charge at a type-$k$ station at location $j$ as $u_{ijk} := \alpha_{ijk}+\beta s_{ij}+\varepsilon_{ijk}$, where $\alpha_{ijk}$ is the attraction level specified by location and type, $s_{ij}$ is the distance between node $i$ and site $j$, $\beta$ is the (negative) impact coefficient for \(s_{ij}\), and $\varepsilon_{ijk}$ is a random noise. EV users can also install home-charging piles and the utility is defined as $u_{i0} := \alpha_{i0}+\varepsilon_{i0}$. Assume that $\varepsilon$ is i.i.d. Gumbel distributed, then in year $t$, the probability of a node-$i$ user charging at a type-\(k\) facility of site \(j\) is

\begin{equation}
    p_{ijk}^{t} = \frac{w_{ijk}x_{jk}^{t}}{w_{i0}+\sum_{j} \sum_{k} w_{ijk}x_{jk}^{t}}, \label{MNL}
\end{equation}
where we denote $w_{ijk}:=\exp(\alpha_{ijk}+\beta s_{ij})$ and $w_{i0}:=\exp(\alpha_{i0})$ for notation brevity.

If the charging demand is deterministic and known, the location problem is formulated as
\begin{subequations}
\begin{align}
    \max_x \ &\sum_{t} \sum_{i} \sum_{j} \sum_{k} r_k d_i^t p_{ijk}^t \label{d_objective} \\
    \text{ s.t. }\ &\sum_{j} \sum_{k} (x_{jk}^t - x_{jk}^{t-1}) c_{jk} \leq B^t, && \forall t, \label{d_budget} \\
    \quad & x_{jk}^{t-1} \leq x_{jk}^t,&& \forall j,k,t, \label{d_notdestroy}\\
    \quad & x_{jk}^t \in \{0,1\},&&\forall j,k,t. \label{d_variable}
\end{align}
\end{subequations}

The objective function \eqref{d_objective} seeks to maximize the total revenue throughout the planning horizon, constraints~\eqref{d_budget} designate the annual construction budget, and constraints~\eqref{d_notdestroy} require that once a charging station is built, it cannot be dismantled or closed in the subsequent years.



\subsection{Multi-stage Stochastic Programming Model}\label{model:multi}
In reality, the growth of charging demands follows a general stochastic process.
In each stage \(t \in T\), the investor locates the charging stations \(x^t\) based on the historical demand data up to stage \(t-1\). Then, she learns the actual demand \(d^t := \{d_i^t, i \in I\}\) of this stage and updates her projection for the future demand. This produces the following multi-stage stochastic programming model:

\begin{align}
    & \max_{x^1} \ \text{E}_{d^1}\Big[ \sum_{i} \sum_{j} \sum_{k} r_k d_i^1 p_{ijk}^1 \tag{{\bf MS}} \label{m_objective} \\
    & + \max_{x^2}\Big\{ \text{E}_{d^2|d^1}\Big[\sum_{i} \sum_{j} \sum_{k} r_k d_i^2 p_{ijk}^2 + \max_{x^3} \Big\{ \ldots \nonumber \\
    & +\max_{x^T}\Big\{\text{E}_{d^T|d^1, \ldots, d^{T-1}}\Big[\sum_{i} \sum_{j} \sum_{k} r_k d_i^T p_{ijk}^T \Big] \Big\} \ldots \Big\} \Big] \nonumber \\
    & \text{ s.t. } \ \text{\eqref{d_budget}--\eqref{d_variable}}. \nonumber
\end{align}
Note that, in model~\eqref{m_objective}, the location decision \(x^t\) can be dynamically adjusted according to the most recent demand data \(d^{t-1}\). If we refrain from such adjustment and make all location decisions before observing demands, however, then it reduces to the following two-stage model:
\begin{align}
    \max_x &\ \text{E}_{d^1, \ldots, d^T}\Big[\sum_{t} \sum_{i} \sum_{j} \sum_{k} r_k d_i^t p_{ijk}^t\Big] \tag{{\bf TS}} \label{2_objective} \\
    \text{ s.t. } & \ \text{\eqref{d_budget}--\eqref{d_variable}}. \nonumber
\end{align}

\section{SOLUTION APPROACH}\label{sec:solve}
Solving~\eqref{m_objective} directly is intractable because (i) there are infinite (uncountable) number of \(d^t\) realizations and (ii) the MNL model~\eqref{MNL} makes the formulation nonlinear and nonconvex. For difficulty (i), we approximate the stochastic process of demand growth using a scenario tree, consisting of a set \(S\) of scenarios. Each scenario \(s \in S\) takes place with probability \(\pi^s\) and characterizes a fully realized demand trajectory \(\{d^{ts}_i: t \in T, i \in I\}\), and to describe a general (discrete) stochastic process, any two scenarios \(r, s\) may share the same realized demand up to a stage \(t \leq T-1\). To this end, we define set \(S(t,s)\) for the set of scenarios that have identical path up to stage \(t\) when compared with scenario \(s\) (including \(s\) itself). This gives rise the following sample average approximation of~\eqref{m_objective}:
\begin{align}
    \max_{x, \ q} \ & \sum_{s} \pi^s \sum_{t} \sum_{i} q^{ts}_{i} \tag{{\bf MS-SAA}} \label{scenario-multi} \\
    \text{s.t.} \ \ & x_h^{tr} = x_h^{ts}, \qquad \qquad \ \forall t,h, (r,s) \in S(t,s),\nonumber\\
    & q^{ts}_i \leq \sum_h \frac{d^{ts}_{ih}w_{ih}x^{ts}_h}{w_{i0}+\sum_{h} w_{ih}x^{ts}_h}, \qquad \ \forall i,t,s, \label{q-def}\\ 
    &\sum_{h} (x_{h}^{ts} - x_{h}^{t-1,s}) c_{h} \leq B^t, \qquad \quad \ \ \forall t,s, \nonumber \\
    & x_{h}^{t-1,s} \leq x_{h}^{ts}, \qquad \qquad \qquad \qquad \ \forall h,t,s, \nonumber\\
    & x_{h}^{ts} \in \{0,1\}, \qquad \qquad \qquad \qquad \ \ \forall h,t,s. \nonumber
\end{align}
The first constraints in~\eqref{scenario-multi} designate the non-anticipative nature of the location policy, ensuring that the locations \(x^{ts}\) in stage \(t\) are determined only based on the demand realizations up to stage \(t-1\). In addition, we combine the indices \((j,k)\) into a single index \(h \in H := J \times K\) and define \(d^{ts}_{ih} := r_k d^{ts}_i\). Accordingly, variable \(q^{ts}_i\) represents the revenue of charging from node \(i\) in time \(t\) and scenario \(s\), and its definition through~\eqref{q-def} inherits the nonlinear and nonconvex MNL model~\eqref{MNL}.

To address the difficulty (ii), the state-of-the-art approach~\cite{borrero2016simple} drew upon 0-1 fractional programming. The first approach, denoted by R1 in~\cite{borrero2016simple}, defines a new variable \(y^{ts}_{i} := 1/(w_{i0}+\sum_{h} w_{ih}x^{ts}_h)\) and rewrites the right-hand side of~\eqref{q-def} as \(\sum_h d^{ts}_{ih}w_{ih}x^{ts}_h y^{ts}_{i}\). Then, R1 introduces McCormick inequalities to linearize the bilinear term \(x^{ts}_h y^{ts}_{i}\). The second approach, denoted by R4 in~\cite{borrero2016simple}, defines a second new variable \(z^{ts}_{i} := (\sum_h d^{ts}_{ih}w_{ih}x^{ts}_h)/(w_{i0}+\sum_{h} w_{ih}x^{ts}_h)\). This trivially linearizes~\eqref{q-def} but introduces bilinear constraints
\[
z^{ts}_{i}\Big(w_{i0}+\sum_{h} w_{ih}x^{ts}_h\Big) \ = \ \sum_h d^{ts}_{ih}w_{ih}x^{ts}_h, \quad \forall i,t,s.
\]
To handle these constraints efficiently, R4 assumes that \(w_{i0}\) and \(w_{ih}\) are integer-valued and encodes $\sum_{h} w_{ih}x^{ts}_h$ using a binary expansion. Then, R4 evokes McCormick inequalities to linearize the ensuing bilinear terms.

Different from these approaches, we propose to recast~\eqref{q-def} as \emph{second-order conic constraints} and do so \emph{without introducing new variables}. This technical advancement maintains the formulation size of~\eqref{scenario-multi} (see Table~\ref{table:summary}), making it scalable to large-scale instances. Numerically, this significantly accelerates the computation of~\eqref{scenario-multi} as compared to R1 and R4 (see Table~\ref{table:compute}).

\subsection{Second-Order Conic Representation}\label{Inequality}

We first present the main technical result of the paper.
\begin{theorem} \label{thm:soc}
The (nonlinear and nonconvex) constraints \eqref{q-def} admit the following second-order conic representation:
\begin{align}
& \left\| \begin{bmatrix}
2\sqrt{\overline{d}^{ts}_i w_{i0}}\\
\{2\sqrt{(\overline{d}^{ts}_i - d^{ts}_{ih}) w_{ih}} x^{ts}_h\}_{h\in H}\\
w_{i0} + \sum_h w_{ih} x^{ts}_h - \overline{d}^{ts}_i + q^{ts}_i
\end{bmatrix} \right\|_2 \nonumber \\[0.5em]
\leq \ \ & w_{i0} + \sum_h w_{ih} x^{ts}_h + \bar{d}^{ts}_i - q^{ts}_i, \tag{\bf SOC} \label{soc}
\end{align}
where \(\overline{d}^{ts}_i := \max_{h \in H}\{d^{ts}_{ih}\}\).
\end{theorem}
\begin{proof}
In view that \(w_{ih} > 0\) for all \(h \in H \cup \{0\}\) and \(x^{ts}_h = (x^{ts}_h)^2\) because \(x^{ts}_h \in \{0,1\}\), we recast~\eqref{q-def} as
$
q^{ts}_i \big(w_{i0}+\sum_{h} w_{ih}x^{ts}_h\big) \leq \sum_h \overline{d}^{ts}_i w_{ih} x^{ts}_h - \sum_h \big(\overline{d}^{ts}_i - d^{ts}_{ih}\big) w_{ih} \big(x^{ts}_h\big)^2.
$
It follows that
\begin{align*}
\text{\eqref{q-def}} \ \Longleftrightarrow \ & \sum_h \big(\overline{d}^{ts}_i - d^{ts}_{ih}\big) w_{ih} \big(x^{ts}_h\big)^2 \leq \sum_h \overline{d}^{ts}_i w_{ih} x^{ts}_h - \\
& q^{ts}_i \Big(w_{i0}+\sum_{h} w_{ih}x^{ts}_h\Big) \\
\Longleftrightarrow \ & \overline{d}^{ts}_i w_{i0} + \sum_h \big(\overline{d}^{ts}_i - d^{ts}_{ih}\big) w_{ih} \big(x^{ts}_h\big)^2 \leq \\ & \big(\overline{d}^{ts}_i - q^{ts}_i \big) \Big(w_{i0}+\sum_{h} w_{ih}x^{ts}_h\Big) \\
\Longleftrightarrow \ & \text{\eqref{soc}}.
\end{align*}
\end{proof}
Theorem~\ref{thm:soc} allows us to replace constraints~\eqref{q-def} with~\eqref{soc}. That is, we can solve~\eqref{scenario-multi} in an off-the-shelf solver (e.g., Gurobi) as a mixed-integer second-order conic program (MISOCP). Notably, we do not need to introduce new variables or assume integer-valued parameters. We compare the formulation size of our approach with the state-of-the-art in Table~\ref{table:summary}.
\begin{table}[htbp]
    \centering
    \caption{Sizes of Various Reformulations for~\eqref{scenario-multi}.\\
    \(B := \sum_i (\lfloor  \log_2(\sum_h|w_{ih}|) \rfloor + 1 )\)}
    \label{table:summary}
    \resizebox{.49\textwidth}{!}{
    \begin{tabular}{c c c c}
    \toprule
        Reformulation & \makecell{Continuous\\ variables} & \makecell{Binary\\ variables} & \makecell{Constraints}\\
    \midrule
        This paper & $|I||T||S|$ & $|H||T||S|$ & $\mathcal{O}\big((|H||S|+|I|)|T||S|\big)$\\
        \cite{borrero2016simple}--R1 & $(|H|+1)|I||T||S|$ & $|H||T||S|$ & $\mathcal{O}\big((|H||I|+|H||S|)|T||S|\big)$\\
        \cite{borrero2016simple}--R4 & $(B+|I|)|T||S|$ & $(B+|H|)|T||S|$ & $\mathcal{O}\big((|H||S|+|I|+|H|+B+1)|T||S|\big)$\\
    \bottomrule
    \end{tabular}
    }
\end{table}

Although commercial solvers can solve MISOCP directly, the numerical performance is in general inferior that of solving mixed-integer \emph{linear} programs (MILPs). This motivates us to seek a linear representation of~\eqref{scenario-multi}, which is fortunately available thanks to Theorem~\ref{thm:soc}.
\begin{theorem} \label{thm:subgradient}
For fixed \(t,s\), define function \(Q^{ts}_i: [0,1]^{|H|} \rightarrow \mathbb{R}\), such that
\[
Q^{ts}_i(x^{ts}) = \sum_h w_{ih} \left(\frac{\overline{d}^{ts}_ix^{ts}_h - \big(\overline{d}^{ts}_i - d^{ts}_{ih}\big) \big(x^{ts}_h\big)^2}{w_{i0}+\sum_{h} w_{ih}x^{ts}_h}\right).
\]
Then, \(Q^{ts}_i\) is concave over \([0,1]^{|H|}\). In addition, constraints~\eqref{q-def} are equivalent to
\begin{equation*}
q^{ts}_i \leq Q^{ts}_i\big(\hat{x}^{ts}\big) + \sum_h \left( \frac{\partial Q^{ts}_i}{\partial \hat{x}^{ts}_h} \right) \big(x^{ts}_h - \hat{x}^{ts}_h\big), \forall \hat{x}^{ts} \in \{0,1\}^{|H|},
\end{equation*}
where \(\partial Q^{ts}_i/\partial \hat{x}^{ts}_h\) is the partial derivative of \(Q^{ts}_i\) at \(\hat{x}^{ts}\).
\end{theorem}
\begin{proof}
First, the concavity of \(Q^{ts}_i\) follows from the fact that its hypograph is second-order conic representable, as shown in Theorem~\ref{thm:soc}. Second, constraints~\eqref{q-def} are equivalent to the hypograph of \(Q^{ts}_i\) because \(x^{ts}_{h} \in \{0,1\}\). Then, they are further equivalent to the intersection of all the subgradient inequalities of \(Q^{ts}_i\) because of the concavity of \(Q^{ts}_i\).
\end{proof}

We remark that \(\partial Q^{ts}_i/\partial \hat{x}^{ts}_h\) can be computed in closed-form for any \(\hat{x}^{ts}\). We call the reformulation of constraints~\eqref{q-def} in Theorem~\ref{thm:subgradient} the subgradient inequalities (SGI), with which \eqref{scenario-multi} can now be solved as a MILP. Although SGI are exponentially many, we can incorporate them on-the-fly, only when needed, as we solve the MILP in a branch-and-bound tree. That is, we ignore all SGI (and hence constraints~\eqref{q-def}) at the beginning. As the branch-and-bound tree evolves, whenever an incumbent solution with an integer-valued \(\hat{x}^{ts}\) is found, we check if the SGI associated with \(\hat{x}^{ts}\) is violated: if \(q^{ts}_i \leq Q^{ts}_i\big(\hat{x}^{ts}\big)\), then \(\hat{x}^{ts}\) is optimal to~\eqref{scenario-multi}; otherwise, we incorporate the SGI into the formulation for the current tree node and all its descendants. In commercial solvers, this can be implemented through \texttt{lazy callback}.

\section{COMPUTATIONAL RESULTS}\label{sec:numerical}
We first outline the experiment settings. Then, we report the solving times of various reformulations for~\eqref{scenario-multi}. Finally, we analyze the performance of~\eqref{2_objective} and~\eqref{m_objective}, highlighting their suitable application contexts.

\subsection{Experiment Settings}
Our numerical experiments 
utilize the ``Distance'' dataset in \cite{lamontagne2023optimising}, which seeks to locate charging stations over a four-year period in Trois-Rivi\`eres, Qu\'{e}bec, Canada. We apply data clustering to obtain \(|I|=58\) user demand nodes and \(|J|=10\) candidate sites, as illustrated in Fig.~\ref{figure1}. In addition, we consider \(|K|=2\) different types of charging stations, with no existing charging infrastructure at the outset. Other necessary data, such as construction costs and revenues, are derived from the reports by Natural Resources Canada~\cite{Canada}.
Based on the statistics of new EV registrations released by the national statistical office \cite{demanddata}, we calculate the growth of EV in Canada and accordingly that of charging demands. We assume that the annual growth rates of charging demands follow the Gaussian distribution, as confirmed by the Shapiro-Wilk test, with an average of 43.5\% and a standard deviation (S.D.) of 9.6\%. Moreover, to construct~\eqref{scenario-multi}, we generated \(|S|=81\) scenarios, with the growth rate at each demand node independently drawn from the aforementioned distribution. All the models are solved in Gurobi 11.0.0 using the default solver configurations.

\begin{figure}[thpb]
      \centering
      \includegraphics[width = 0.4\textwidth]{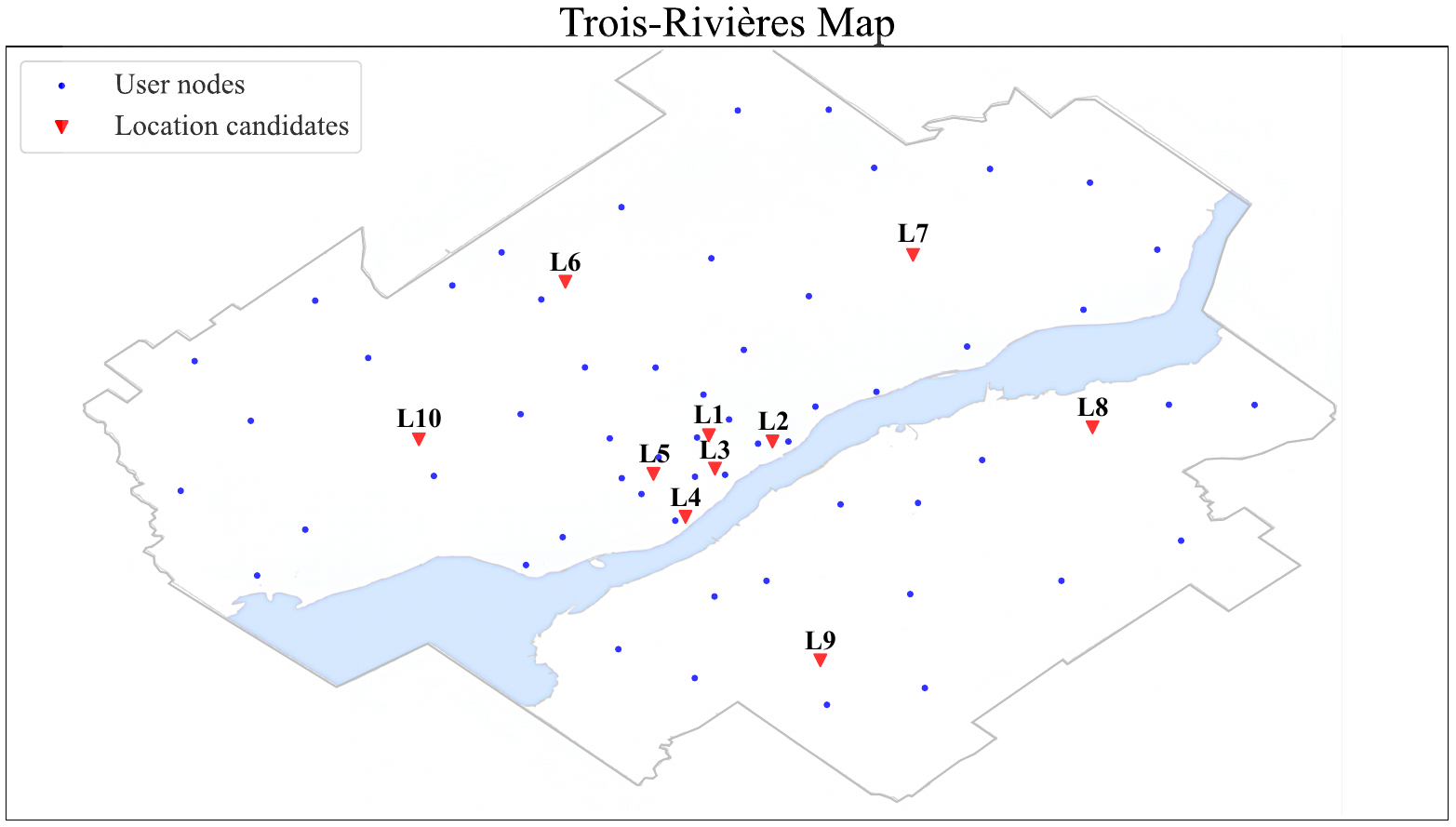}
      \caption{User demand nodes and candidate sites in Trois-Rivi\`eres.}
      \label{figure1}
\end{figure}

\begin{table}[htbp]
    \renewcommand{\arraystretch}{1.1}
    \setlength{\tabcolsep}{4pt}
    \centering
    \caption{Comparison among various reformulations}
    \label{table:compute}
    \begin{tabular}{c c c c c c}
    \toprule
         & Approach & \makecell{Continuous\\ variables} & \makecell{Binary\\ variables} & \makecell{Constraints} & CPU time\\
    \midrule
         \multirow{3}*{\eqref{2_objective}} & SGI & 232 & 80 & 316 & 1.33s\\ 
          & \cite{borrero2016simple}--R1 & 4,872 & 80 & 18,876 & 8.69s \\
          & \cite{borrero2016simple}--R4 & 2,356 & 2,204 & 9,403 & 155s\\
    \midrule
        \multirow{3}*{\eqref{m_objective}} & SGI & 18,792 & 6,480 & 29,656 & 3,226s\\
        &\cite{borrero2016simple}--R1 & 394,632 & 6,480 & 1,528,956 & 5,059s \\ 
          & \cite{borrero2016simple}--R4 & 190,836 & 178,524 & 732,483 & 47,347s\\
    \bottomrule
    \end{tabular}
\end{table}

\subsection{Comparison of Solving Time}
We compare our approach (SGI) with the state-of-the-art approaches R1 and R4 in~\cite{borrero2016simple}.
Table \ref{table:compute} compares the formulation size and CPU time of these approaches. From this table, we observe that, for both~\eqref{2_objective} and~\eqref{m_objective}, SGI outperforms R1 and R4 in solving time significantly. This can be explained by its smaller formulation size. In addition, our approach incorporates the SGI on-the-fly as lazy cuts, further reducing the computational burden. This empirical results confirm the theoretical comparison in Table~\ref{table:summary}.

\subsection{Comparison of~\eqref{2_objective} and~\eqref{m_objective} Results}
The result of~\eqref{m_objective} can concentrate in areas with higher demand, thereby revealing significant dynamics. We take two scenarios to illustrate this feature as shown in Fig. \ref{figure-year3}. On these maps, the size and intensity of the blue circles indicate the level of demand at each user node: larger and darker circles signify higher demand. The planning decisions remain the same across the first two years and the dynamics begin to manifest in the third year. In scenario A, the demand growth is relatively higher in the southern region, leading to the planning of new charging stations at locations L4 and L9 in year 3. In comparison, Scenario B witnesses a more substantial demand within the central area and new stations are concentrated around the center (L3, L5). On average, adopting~\eqref{m_objective} results in 2.26\% higher objective value compared to the upfront, one-time planning of~\eqref{2_objective}.

\begin{figure}[thpb]
      \centering
      \includegraphics[width = 0.45\textwidth]{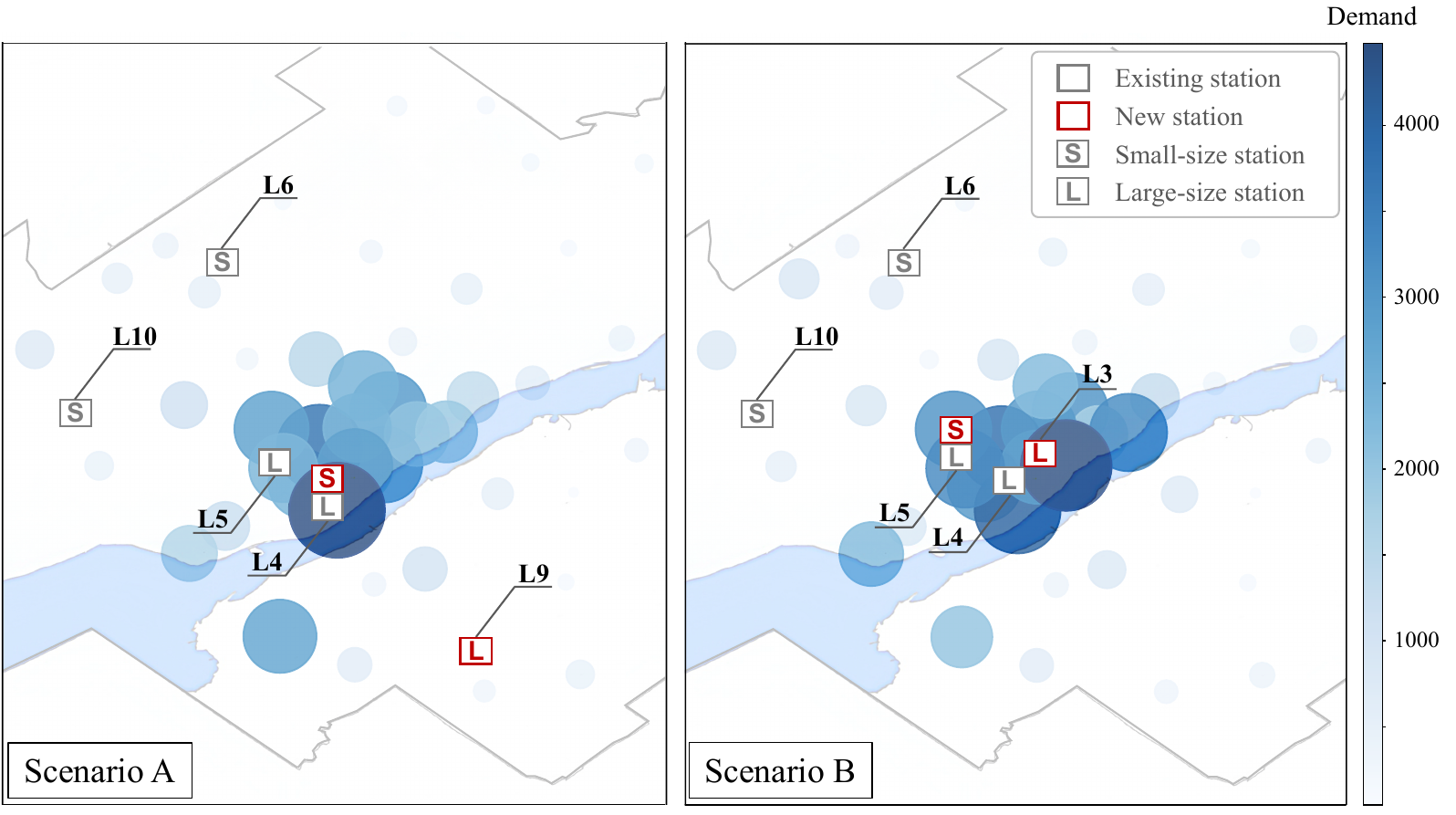}
      \caption{Decision dynamics in year 3 of scenario A and B, with A showing more demand in the south and B in the center.}
      \label{figure-year3}
\end{figure}

To demonstrate the advantages of~\eqref{m_objective} more clearly, we assume another setting: the city develops new sub-centers during the planning period. Two scenarios are picked to illustrate~\eqref{m_objective} decisions, as shown in Fig. \ref{figure-nands}. In scenario C, new city center emerges in the northwest. Thus, both large-scale and small-scale stations are built in L6 and L10. In scenario D, new city center emerges in the southeast and L9 receives more attention. However, in~\eqref{2_objective}, due to the use of average value across scenarios, more stations are established in the original center (see Fig. \ref{figure-middle}). The revenue gap between the two models expands to 9.40\%.
This highlights the multi-stage model's ability to capture demand dynamics and the advantage of flexible decision-making.

\begin{figure}[thpb]
      \centering
      \includegraphics[width = 0.45\textwidth]{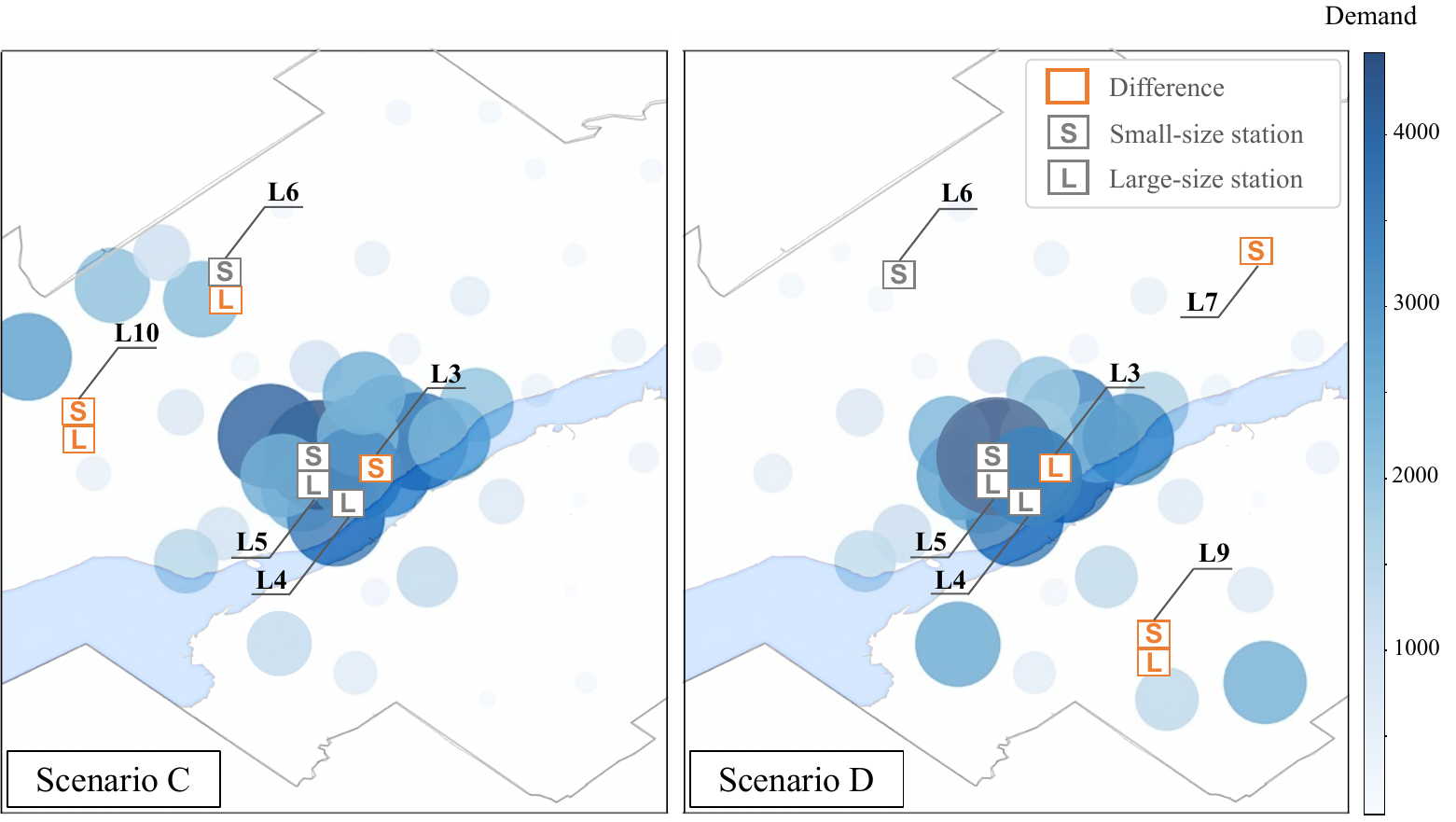}
      \caption{Planning under scenario C and D, with C showing a new urban center in the northwest and D in the southeast.}
      \label{figure-nands}
\end{figure}
\begin{figure}[thpb]
      \centering
      \includegraphics[width = 0.45\textwidth]{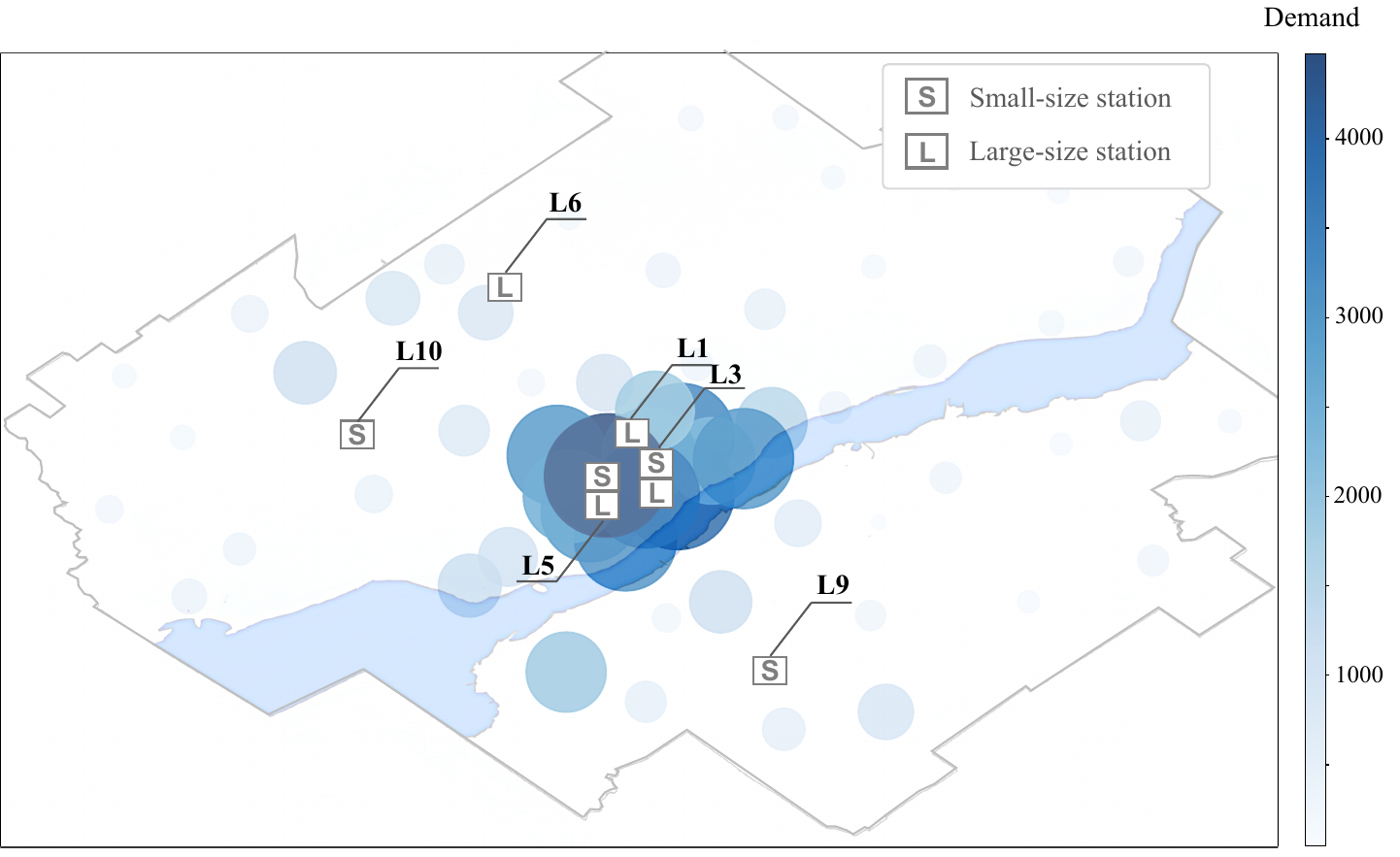}
      \caption{Planning of two-stage model when new urban centers may emerge.}
      \label{figure-middle}
\end{figure}

\subsection{Sensitivity Analysis}
\subsubsection{Variability of demand growth}
To assess the impact of growth variability on model performance, we conduct a series of experiments by increasing the S.D. of growth rate. Table \ref{table:sigma} summarizes the objective gap between~\eqref{2_objective} and~\eqref{m_objective} with S.D. varying from 9.6\% to 30\%. The advantage of~\eqref{m_objective} amplifies as the variability in the growth rate increases. This prescribes that, in situations where the EV market is subject to significant influences (e.g., from policy shifts, technological progress, or economic variations), \eqref{m_objective} produces a more promising location policy than~\eqref{2_objective}.

\begin{table}[htbp]
    \centering
    \setlength{\tabcolsep}{6pt}
    \caption{Impact of growth variability on model performance}
    \label{table:sigma}
    \begin{tabular}{c c c c c c}
    \toprule
        S.D. of growth rate & 9.6\% & 15\% & 20\% & 25\% & 30\%\\
    \midrule
         Objective gap & 2.26\% & 2.82\% & 3.43\% & 4.11\% & 5.21\%\\
    \bottomrule
    \end{tabular}
\end{table}

\subsubsection{Distance-to-charge}
is a key factor determining which station users choose to patronise. Here, we explore the performance difference between~\eqref{2_objective} and~\eqref{m_objective} when users emphasize distance differently. We conduct experiments on the ``Simple'' datasets of \cite{lamontagne2023optimising}, where the coefficient $\beta$ is reduced to 1/10 of its original value as shown in Table \ref{table:distance}. Under the new settings, the gap narrows to merely 0.33\%, indicating that it is indifferent to apply~\eqref{2_objective} or~\eqref{m_objective} when users become much more willing to travel to charge. This makes sense because when distance-to-charge becomes less of a concern for users, it makes less difference where to locate the charging stations. Accordingly, the flexibility provided by~\eqref{m_objective} devalues. Hence, in contexts where the distance-to-charge is not a significant concern (e.g., for rural EV users), both~\eqref{2_objective} and~\eqref{m_objective} are viable options.

\begin{table}[htbp]
    \centering
    \setlength{\tabcolsep}{12pt}
    \caption{Impact of distance coefficient}
    \label{table:distance}
    \begin{tabular}{c c c}
    \toprule
        Dataset & Distance coefficient $\beta$ & Objective gap\\
    \midrule
        Distance & -0.63 & 2.26\% \\
        Simple & -0.063 & 0.33\% \\
    \bottomrule
    \end{tabular}
\end{table}

\subsubsection{Demand growth heterogeneity}
The growth rate of EV users may vary across different areas. We divide the urban area into two categories: central and suburb. By adjusting the growth rates for these two regions separately, we explore how their difference affects the performance difference of~\eqref{2_objective} and~\eqref{m_objective} in Table~\ref{table:mu}. Therein, the first row is the percentage by which the average growth rate in the suburbs surpasses that in the center. These adjustments determine whether demands become more concentrated or spread out, which is reported in the third row (the proportion at the end of the planning horizon). Our findings reveal that when the bulk of the demand is confined to the smaller central area, it is optimal to locate among the five central locations (L1--L5), reducing the diversity of planning. Consequently, the advantage of using~\eqref{m_objective} over~\eqref{2_objective} decreases swiftly. Therefore, \eqref{m_objective} is more advantageous when demands are more evenly distributed across regions.

\begin{table}[htbp]
    \centering
    \caption{Impact of growth rate heterogeneity}
    \setlength{\tabcolsep}{4pt}
    \label{table:mu}
    \begin{tabular}{c c c c c c}
    \specialrule{0.1em}{1.5pt}{1.5pt}
        Growth rate difference & -20\% & -10\% & 0\% & 10\% & 20\%\\
    \specialrule{0.03em}{1pt}{2pt}
         Objective gap & 0.57\% & 1.24\% & 2.26\% & 2.56\% & 2.74\%\\
    \specialrule{0em}{2pt}{3pt}
         \makecell{Proportion of central\\area demand} & 77.64\% & 70.67\% & 64.36\% & 57.51\% & 50.78\%\\
    \specialrule{0.1em}{1pt}{1pt}
    \end{tabular}
\end{table}

\section{CONCLUSIONS}\label{sec:conclusion}
We studied sequential location of charging stations considering stochastic charging demand growth. We incorporated EV user choice behavior using MNL and formulated two-stage and multi-stage stochastic optimization models. We derived SOC representations for the models and designed solution approaches using SGI, which significantly accelerates the solution process. Through numerical experiments using the Trois-Rivi\`eres datasets, we confirmed the superiority of our proposed model and solution approaches. The multi-stage model proved to be promising in generating higher revenues. Sensitivity analyses further showed that~\eqref{m_objective} becomes particularly prominent under conditions of high demand growth rate variability, increased demand heterogeneity, and high user sensitivity to the distance-to-charge. 

\addtolength{\textheight}{-0cm}   










\end{document}